\documentclass[preprint,12pt]{elsarticle}

\usepackage{amssymb}
\usepackage{amsthm}
\usepackage{amssymb}

\usepackage{amsmath,amssymb}
\usepackage{natbib}

\newtheorem{theorem}{Theorem}

\newtheorem{lemma}{Lemma}

\newtheorem{remark}{Remark}

\graphicspath{ {figures/} }

\journal{Journal of Computational and Applied Mathematics
}

\begin{document}

\begin{frontmatter}



\title{Finite difference and numerical differentiation: General formulae from deferred corrections\tnoteref{label1}}
\tnotetext[label1]{The authors would like to acknowledge the financial support of the Discovery Grant Program of the Natural Sciences and Engineering Research Council of Canada (NSERC) and a scholarship to the first author from the NSERC CREATE program ``G\'enie par la Simulation''.}
\author{Saint-Cyr E.R.\ Koyaguerebo-Im\'e}
\ead{skoya005@uottawa.ca}
\author{Yves Bourgault}
\ead{ybourg@uottawa.ca}

\cortext[cor1]{Corresponding author: ybourg@uottawa.ca}
\address{Department of Mathematics and Statistics,
         University of Ottawa, STEM Complex, \\
         150 Louis-Pasteur Pvt, Ottawa, ON, Canada, K1N 6N5, Tel.: +613-562-5800x2103 
}

\begin{abstract}
This paper provides a new approach to derive various arbitrary high order finite difference formulae for the numerical differentiation of analytic functions. In this approach, various first and second order formulae for the numerical approximation of analytic functions are given with error terms explicitly  expanded as Taylor series of the analytic function. These lower order approximations are successively improved by one or two (two order improvement for centered formulae) to give finite difference formulae of arbitrary high order. The new approach allows to recover the standard backward, forward, and centered finite difference formulae which are given in terms of formal power series of finite difference operators. Examples of new formulae suited for deferred correction methods are given. 
\end{abstract}




\end{frontmatter}


\section{Introduction}
\label{intro}
Finite differences are commonly used for discrete approximations of derivatives. Large classess of schemes for the numerical approximation of ordinary differential equations (ODEs) and partial differential equations (PDEs) are derived from finite differences. Formulae for numerical differentiations are generally obtained from a linear combination of Taylor series, which leads to solving a system of linear equations \cite{khan1999closed,khan2000new,khan2003taylor,quarteroni2010} or calculating derivatives of  interpolating polynomials (for instance see \cite{quarteroni2010}). References \cite{hildebrand1974introduction,chung2010computational,dahlquist2008numerical} give a number of finite difference formulae, for high order approximation of derivatives, in term of formal power series of finite difference operators. 
 
The purpose of this paper is to provide some basic results on finite difference approximations, which results are required for the numerical analysis of higher order time-stepping schemes for ODEs and PDEs. We introduce a new approach to derive arbitrary high order finite difference formulae which avoids the need for solving a system of linear equations. We provide various formulae for the discrete approximation of any order $p$ derivative of an analytic function $u$ at a point $t_0$ using $p$ arbitrary points  $t_1<t_2<\cdots<t_p$ evenly spread around $t_0$. These discrete approximations are of order 1 or 2 (order 2 for centred formulae), with errors explicitly expanded in terms of Taylor series with the derivatives $u^{(p+i)}(t_0)$, $i=1,2,\cdots $. Substituting successively $u^{(p+1)}(t_0), u^{(p+2)}(t_0), \cdots $ by their finite difference approximations in the error term for the discrete approximation of $u^{(p)}(t_0)$, we improve successively by 1 or 2 the order of the discrete approximation of $u^{(p)}(t_0)$. An efficient choice of the discrete points minimizes the number of points needed for a given order of accuracy of the discrete approximation of $u^{(p)}(t_0)$. Our approach can be used to recover the existing finite difference formulae, but it also provides various new  formulae. We give three new finite difference formulae which are useful for the construction of new high order time-stepping schemes and their efficient starting procedures via the deferred correction (DC) method. In fact, the use of standard backward and central finite differences in building high order time-stepping schemes via the DC method leads to the computation of starting values for these schemes outside the solution interval while the standard forward finite difference formula leads to unstable schemes (see, e.g., \cite{daniel1967interated,MR2058857,kress2002deferred,koyaguerebo2019arbitrary,koyaguerebo2020unconditionally}). 

The paper is organized as follows: in section 2 we recall the main finite difference operators and prove some of their main properties; section 3 presents general first and second order approximations of derivatives with error terms explicitly expressed as Taylor series; section 4 gives many results for arbitrary high order finite difference approximations, and section 5 deals with a numerical test. 
%
%
%
\section{Properties of finite difference operators}
\label{sec:1}
In this section we recall the standard finite difference operators and provided some of their useful properties.

For a given spacing $k>0$ and a real $t_0\in \mathbb{R}$, we denote $t_n=nk$ and $t_{n+1/2}=(n+1/2)k$, for each integer $n$. The centered, forward and backward difference operators $D$, $D_+$ and $D_-$, respectively, related to $k$, and applied to a function $u$ from $\mathbb{R}$ into a Banach space $X$, are defined as follows:
$$
Du(t_{n+1/2})=\frac{u(t_{n+1})-u(t_n)}{k},$$
$$D_+u(t_{n})=\frac{u(t_{n+1})-u(t_n)}{k},$$ 
and 
$$D_-u(t_{n})=\frac{u(t_{n})-u(t_{n-1})}{k}.$$ The average operator is denoted by $E$:
$$E u(t_{n+1/2})=\widehat{u}(t_{n+1})=\frac{u(t_{n+1})+u(t_n)}{2}.$$
The composites of $D_+$ and $D_-$ are defined recursively. They commute, that is 
$$(D_+D_-)u(t_n)=(D_-D_+)u(t_n)=D_-D_+u(t_n),$$ and satisfy the identities
\begin{equation}
\label{bb1}
(D_+D_-)^mu(t_n)=k^{-2m} \sum_{j=0}^{2m}(-1)^j {{2m}\choose {j}}u(t_{n+m-j} ),
\end{equation}
\begin{equation}
\label{bb2}
D_-(D_+D_-)^mu(t_n)=k^{-2m-1}\sum_{j=0}^{2m+1}(-1)^j{{2m+1}\choose {j}}u(t_{n+m-j}),
\end{equation}and
\begin{equation}
\label{bb3}
D_+^{m_1}D_-^{m_2}u(t_n)=k^{-m_1-m_2}\sum_{j=0}^{m_1+m_2}(-1)^j{{m_1+m_2}\choose {j}}u(t_{n+m_1-j}),
\end{equation}
for each nonnegative integer $m$, $m_1$, and $m_2$ such that these sums exist. Formulae (\ref{bb1})-(\ref{bb3}) can be proven by a straightforward induction argument. 

We introduce the double index $\alpha ^m=(\alpha ^m_1,\alpha ^m_2) \in  \left\lbrace 0,1,...,m\right\rbrace \times \left\lbrace 0,1,...,m\right\rbrace $ such that
\begin{equation}
\label{aa3}
D^{\alpha^m}u(t_n)=D_+^{\alpha_1^m}D_-^{\alpha_2^m}u(t_n).
\end{equation}  

\begin{remark} 
\label{rmk:1}
If $\vert \alpha^m\vert=\alpha_1^m+\alpha_2^m$ is even, then we have
\begin{equation}
\label{a4}
D^{\alpha^m}u(t_n)=(D_+D_-)^{|\alpha^m|/2}u(t_{m'}),
\end{equation}for some integer $m'$. For example, $$D_+D_-^3u(t_n)=(D_+D_-)^2u(t_{n-1}),$$ and  $$D_-^4u(t_n)=(D_+D_-)^2u(t_{n-2}).$$
\end{remark}

\begin{theorem}[Finite difference approximation of a product]Suppose that $X$ is a Banach algebra. Then, for any functions $f,g:\mathbb{R} \rightarrow X$, we have
\begin{equation}
\label{aa9}
D_-(fg)(t_n)=D_-f(t_n)g(t_n)+f(t_n)D_-g(t_n)-kD_-f(t_n)D_-g(t_n),
\end{equation}
\begin{equation}
\label{aa10}
D_+(fg)(t_n)=D_+f(t_n)g(t_n)+f(t_n)D_+g(t_n)+kD_+f(t_n)D_+g(t_n),
\end{equation}and
\begin{equation}
\label{aa11}
\begin{aligned}		
D_+D_-(fg)(t_n)=&D_+D_-f(t_n)g(t_n)+f(t_n)D_+D_-g(t_n)+D_+f(t_n)D_-g(t_n)\\&+D_-f(t_n)D_+g(t_n)+k^2D_+D_-f(t_n)D_+D_-g(t_n).
\end{aligned}
\end{equation}
 More generally,  for each integer $m=1,2,..., $ such that $(D_+D_-)^m(fg)(t_n)$ exists, we have the formula
	
\begin{equation}
\label{aa12}
(D_+D_-)^m(fg)(t_n)=\sum_{j=0}^m\binom{m}{j}k^{2j}\sum_{\alpha^m+\beta^m=(m+j,m+j)}D^{\alpha^m}f(t_n)D^{\beta ^m}g(t_n).
\end{equation}
\end{theorem}

\begin{proof}
The formulae (\ref{aa9})-(\ref{aa11}) can be obtained by a straightforward calculation, so we just need to establish (\ref{aa12}). We proceed by induction on the positive integer $m$. From the index notation introduced in (\ref{aa3}), we can write 
$$
\begin{aligned}	D_+D_-f(t_n)g(t_n)+f(t_n)D_+D_-g(t_n)+D_+f(t_n)D_-g(t_n)+D_-f(t_n)D_+g(t_n)\\=\sum_{\alpha^1+\beta^1=(1,1)}D^{\alpha^1}f(t_n)D^{\beta^1}g(t_n),
\end{aligned}$$
and
$$D_+D_-f(t_n)D_+D_-g(t_n)=D^{\alpha^1}f(t_n)D^{\beta^1}g(t_n),~~\mbox{ with } \alpha^1+\beta^1=(2,2).$$ These two identities combined with (\ref{aa11}) yield 

$$D_+D_-(fg)(t_n)=\sum_{j=0}^1\binom{1}{j}k^{2j}\sum_{\alpha^1+\beta^1=(1+j,1+j)}D^{\alpha^1}f(t_n)D^{\beta ^1}g(t_n),$$that is formula (\ref{aa12}) holds for $m=1$. Now suppose that (\ref{aa12}) holds until  some rank $m\geq 1$. We are going to show that it remains true for $m+1$. By the induction hypothesis, we can write
\begin{equation}
\label{aa13}
(D_+D_-)^{m+1}(fg)(t_n)=\sum_{j=0}^m\binom{m}{j}k^{2j}\sum_{\alpha^m+\beta^m=(m+j,m+j)}D_+D_-[D^{\alpha^m}f(t_n)D^{\beta ^m}g(t_n)].
\end{equation}
Expanding $D_+D_-[D^{\alpha^m}f(t_n)D^{\beta ^m}g(t_n)]$ as in the formula (\ref{aa11}), we deduce that 
\begin{align}
\label{aa14}
	\sum_{\alpha^m+\beta^m=(m+j,m+j)}D_+D_-[D^{\alpha^m}f(t_n)D^{\beta ^m}g(t_n)]=S(j)+k^2S(j+1),
	\end{align}where
	$$S(j)=\sum_{\alpha^{m+1}+\beta^{m+1}=(m+1+j,m+1+j)}D^{\alpha^{m+1}}f(t_n)D^{\beta ^{m+1}}g(t_n).$$
We have
$$
\begin{aligned}
\sum_{j=0}^m\binom{m}{j}k^{2j}[S(j)+&k^2S(j+1)]=S(0)\\&+\sum_{j=1}^mk^{2j}\left[ {{m}\choose{j-1}}+{{m}\choose{j}}\right] S(j)+k^{2m+2}S(m+1),
\end{aligned}$$	
and deduce from (\ref{aa13}), (\ref{aa14}) and the identity $\binom{m}{j}+\binom{m}{j-1}=\binom{m+1}{j}$ that the formula (\ref{aa12}) holds for $m+1$. Finally, we conclude by induction that this formula is true for each suitable positive integer $m$. 
\end{proof}

\begin{theorem}[Finite difference approximation of a composite]
\label{prop:3}
Consider two functions $f$ and $u$ with values into Banach spaces such that the composite $f\circ u$ is defined on $\mathbb{R}$ and the differential $df$ is integrable. Then
\begin{equation}
\label{a10}
D_-f(u(t_n))=\int_0^1df\left(  u(t_{n-1})+\tau k D_-u(t_n)\right) ( D_-u(t_n) )d\tau
\end{equation}
and
\begin{equation}
\label{a11}
D_+f(u(t_n))=\int_0^1df \left( u(t)+\Delta t D_+u(t)\tau \right) ( D_+u(t) )d\tau
\end{equation}
\end{theorem}

\begin{proof}
As in standard mean value theorem. 
\end{proof}

\section{First and second order discrete approximation of derivatives }
\label{sec:2}
In this section  we provide various formulae for the finite difference approximation of arbitrarry high order derivatives of analytic functions. The approximations are of order one or two, and the error terms are explicitly expanded i terms of Taylor series. We need the following lemma which proof is an easy induction.

\begin{lemma}
\label{lem:a8}
For positive integers $m$ and $p$ and for any real $r$, we have
\begin{equation}
\label{a12}
\sum_{j=0}^{m}(-1)^j\binom{m}{j}(m+r-j)^{p}=\left\lbrace \begin{tabular}{cccc}0,& \mbox{ if $1\leq p< m$,}&\\
m!,& \mbox{ if $p=m$.}&
\end{tabular}
\right. 
\end{equation}
In particular, for any nonnegative integer $p$, we have
\begin{equation}
\label{a13}
\sum_{j=0}^{2m}(-1)^j\binom{2m}{j}(m-j)^{2p+1}=0, 
\end{equation}

\begin{equation}
\label{a14}
\sum_{j=0}^{2m+1}(-1)^j\binom{2m+1}{j}(m-j+1/2)^{2p}=0, 
\end{equation}and
\begin{equation}
\label{a15}
\sum_{j=0}^{2m}(-1)^j\binom{2m}{j}\left[  (m-j+1/2)^{2p+1}+(m-j-1/2)^{2p+1}\right] =0.
\end{equation}
\end{lemma} 

\begin{theorem}
\label{theo:a9}
Suppose that the function $u:[0,T] \rightarrow X$ is analytic. Let $0=t_0<t_1<...<t_N=T$, $t_n=nk$, be a partition of the interval $[0,T]$. For each positive integer $m$, we have
\begin{equation}
\label{ab28}
\begin{aligned}
 D_+^mu(t_n) =u^{(m)}(t_n)+\sum_{i=m+1}^{\infty} \frac{k^{i-m}}{i!}u^{(i)}(t_{n})\sum_{j=0}^{m}(-1)^j\binom{m}{j}(m-j)^{i},
\end{aligned}
\end{equation}

\begin{equation}
\label{ab29}
\begin{aligned}
 D_-^mu(t_n) =u^{(m)}(t_n)+\sum_{i=m+1}^{\infty} \frac{k^{i-m}}{i!}u^{(i)}(t_{n})\sum_{j=0}^{m}(m-1)^j\binom{m}{j}(-j)^{i},
\end{aligned}
\end{equation}

\begin{equation}
\label{a25}
\begin{aligned}
 &D_-(D_+D_-)^mu(t_n) =u^{(2m+1)}(t_n)\\&+\sum_{i=2m+2}^{\infty} \frac{k^{i-2m-1}}{i!}u^{(i)}(t_{n})\sum_{j=0}^{2m+1}(-1)^j\binom{2m+1}{j}(m-j)^{i},
\end{aligned}
\end{equation}

\begin{equation}
\label{a26}
(D_+D_-)^m u(t_n)=u^{(2m)}(t_n)+\sum_{i=m+1}^{\infty} \frac{k^{2i-2m}}{(2i)!}u^{(2i)}(t_{n})\sum_{j=0}^{2m}(-1)^j\binom{2m}{j}(m-j)^{2i},
\end{equation}

\begin{equation}
\label{a27}
\begin{aligned}
& D(D_+D_-)^m u(t_{n+1/2})=u^{(2m+1)}(t_{n+1/2})\\&
+\sum_{i=m+1}^{\infty} \frac{k^{2i-2m}}{(2i+1)!}u^{(2i+1)}(t_{n+1/2})\sum_{j=0}^{2m+1}(-1)^j\binom{2m+1}{j}(m-j-1/2)^{2i+1},
\end{aligned}
\end{equation}and
\begin{equation}
\label{a28}
(D_+D_-)^mEu(t_{n+1/2})=u^{(2m)}(t_{n+1/2})+\sum_{i=m+1}^{\infty}a_{mi} \frac{k^{2i-2m}}{(2i)!}u^{(2i)}(t_{n+1/2}),
\end{equation}where
$$a_{mi}=\frac{1}{2}\sum_{j=0}^{2m}(-1)^j\binom{2m}{j}\left[ (m-j+1/2)^{2i}+(m-j-1/2)^{2i}\right].$$
\end{theorem}

\begin{proof} We only prove formula (\ref{a27}). The other formulae can be proven similarly. By Taylor expansion series we have
$$u(t_{n+m-j})=u(t_{n+s})+\sum_{i=1}^{\infty}\frac{k^i}{i!}(m-s-j)^iu^{(i)}(t_{n+s}).$$
Choosing $s=1/2$ in this formula, we deduce from (\ref{bb2}) that
\begin{equation*}
\begin{aligned}
&D(D_+D_-)^mu(t_{n+1/2})=k^{-2m-1}\sum_{j=0}^{2m+1}(-1)^j\binom{2m+1}{j}u(t_{n+m-j})\\&
=k^{-2m-1}\sum_{i=1}^{\infty}\frac{k^i}{i!}u^{(i)}(t_{n+1/2})\sum_{j=0}^{2m+1}(-1)^j\binom{2m+1}{j}(m-j-1/2)^i,
\end{aligned}
\end{equation*}
and (\ref{a27}) follows from (\ref{a12}) and (\ref{a14}).
\end{proof}

\begin{theorem}
\label{cor:a10}
Let $u$ be $C^{m}([0,T],X)$, $m=1,2,...$, and $0=t_0<t_1<...<t_N=T$, $t_n=nk$, be a partition $[0,T]$. Let $m_1$ and $m_2$ be two positive integers such that $m_1+m_2\leq m$.  Then, for each integer $n$ such that $m_2\leq n\leq N-m_1$, $D_+^{m_1}D_-^{m_2}u(t_n)$ is bounded independently of $n$, and we have the estimate
\begin{equation*}
\left \| D_+^{m_1}D_-^{m_2}u(t_n)\right \| \leq  C\max_{t_{n-m2}\leq t\leq t_{n+m_1}}\left \|u^{(m_1+m_2)}(t) \right \|,
\end{equation*} where $C$ is a constant depending only on the integer $m$.
\end{theorem}

\begin{proof} According to Remark \ref{rmk:1}, it is enough to just prove the theorem for  $(D_+D_-)^pf(t_n)$ or $D_-(D_+D_-)^pf(t_n)$, for suitable positive integer $p$ (the case $p=0$ is trivial). As in the previous proof, Taylor expansion of order $(2p-1)$ with integral remainder together with formulae (\ref{bb1}) and (\ref{a12}) yields

$$(D_+D_-)^pu(t_n)=\sum_{j=0}^{2p}\frac{(-1)^j}{(2p-1)!}\binom{2p}{j}(p-j)^{2p}\int_0^1(1-s)^{2p-1}u^{(2p)}(t_n+(p-j)ks)ds.$$ It follows that
$$
\begin{aligned}
\left \| (D_+D_-)^pu(t_n) \right \| &\leq \frac{1}{(2p)!}\sum_{j=0}^{2p}\binom{2p}{j}(p-j)^{2p}\max_{t_{n-p}\leq t\leq t_{n+p}}\left \|u^{(2p)}(t)\right \|.
\end{aligned}$$ Similar reasoning can be applied in the case of $D_-(D_+D_-)^pu(t_n)$.
 
\end{proof}
%
%
%
%
\section{Arbitrary high order finite difference approximations}
\label{sec:3}
\begin{theorem}
\label{theo:1} There exists a sequence $\displaystyle \left\lbrace c_i \right\rbrace_{i\geq 2}$ of real numbers such that for any function $u\in C^{2p+3}\left( [0,T],X\right)$, where $p$ is a positive integer,  and a partition $0=t_0<t_1<...<t_N=T$, $t_n=nk$, of $[0,T]$, we have
\begin{equation}
\label{b6} 
u'(t_{n+1/2})=
\frac{u(t_{n+1})-u(t_n)}{k}-\sum_{i=1}^pc_{2i+1}k^{2i}D(D_+D_-)^iu(t_{n+1/2}) +O(k^{2p+2}),
\end{equation}
and
\begin{equation}
\label{b7} 
u(t_{n+1/2}) =\frac{u(t_{n+1})+u(t_n)}{2}-\sum_{i=1}^pc_{2i}k^{2i}(D_+D_-)^iE u(t_{n+1/2})+O(k^{2p+2}),
\end{equation} for $p\leq n\leq N-1-p$. The error constants for the formulae (\ref{b6}) and (\ref{b7}) are, respectively, $c_{2p+3}$ and $c_{2p+2}$. Table \ref{tab:1} gives the first ten coefficients $c_i$.

\begin{table}[h!] 
\caption{Ten first coefficients of central difference approximations (\ref{b6}) and (\ref{b7})} \centering     
\label{tab:1}
\begin{tabular}{llllllllll}  
\hline\noalign{\smallskip} 
    $c_2$ &~$c_3$ &~~~ $c_4$ &~~~~$c_5$& ~$c_6$&~~$c_7$&~~~~$c_8$&~~~~~$c_9$&~~~$c_{10}$&~~~$c_{11}$\\[1.5ex]
    $\frac{1}{8}$        &$\frac{1}{24}$         &$-\frac{18}{4!2^5}$      &$-\frac{18}{5!2^5}$     &$\frac{450}{6!2^7}$      &$\frac{450}{7!2^7}$      &$-\frac{22050}{8!2^9}$   &$-\frac{22050}{9!2^9}$  &$\frac{1786050}{10!2^{11}}$     & $\frac{1786050}{11!2^{11}}$\\[1ex]
\noalign{\smallskip}\hline
\end{tabular} 
\end{table} 
\end{theorem}

\begin{proof} By Taylor expansion we can write
\begin{equation}
\label{a21}
u(t_{n+1})=u(t_n)+ku'(t_{n+1/2})+\sum_{i=1}^p \frac{d_{1,2i+1}}{(2i+1)!}k ^{2i+1} u^{(2i+1)}(t_{n+1/2}) +O(k^{2p+3})
\end{equation}
and  
\begin{equation}
\label{a22}
u(t_{n+1})=-u(t_n)+2u(t_{n+1/2})+\sum_{i=1}^p\frac{d_{1,2i}}{(2i)!}k ^{2i} u^{(2i)}(t_{n+1/2}) +O(k^{2p+2}),
\end{equation}
 with $d_{1,i}=2^{1-i}$, for $i=2,3,..., 2p+1$. Therefore, substituting successively the derivatives $u^{(3)}(u_{n+1/2})$, $u^{(5)}(t_{n+1/2})$, ... and $u^{(2)}(t_{n+1/2})$, $u^{(4)}(t_{n+1/2})$, ... by their expansion given by the formulae (\ref{a27}) and (\ref{a28}), respectively, into (\ref{a21}) and (\ref{a22}), we deduce the identities
\begin{align*}
&u(t_{n+1})=u(t_n)+ku'(t_{n+\frac{1}{2}})+\frac{d_{1,3}}{3!} k^3DD_+D_-u(t_{n+\frac{1}{2}})+...+\\&\frac{d_{q,2q+1}}{(2q+1)!} k^{2q+1}D(D_+D_-)^qu(t_{n+\frac{1}{2}})+\sum_{i=q+1}^p \frac{d_{q+1,2i+1}}{(2i+1)!} k ^{2i+1}u^{(2i+1)}(t_{n+\frac{1}{2}}) +O(k^{2p+3})
\end{align*}  
and  
\begin{align*}
&u(t_{n+1})=-u(t_n)+2u(t_{n+1/2})+\frac{d_{1,2}}{2!} k^2D_+D_-Eu(t_{n+1/2})+...\\&+\frac{d_{q,2q}}{(2q)!} k^{2q}(D_+D_-)^qEu(t_{n+1/2})+\sum_{i=q+1}^p\frac{d_{q+1,2i}}{(2i)!} k ^{2i} u^{(2i)}(t_{n+1/2}) +O(k^{2p+2})
\end{align*}
where, for $q=1,...,p-1$, and $i=q+1,q+2,...,p$, we have 
\begin{equation*}
d_{q+1,2i+1}=d_{q,2i+1}-\frac{d_{q,2q+1}}{(2q+1)!}\sum_{j=0}^{2q+1}(-1)^j{{2q+1}\choose{j}}(q-j-1/2)^{2i+1},
\end{equation*}and
\begin{equation*}
d_{q+1,2i}=d_{q,2i}-\frac{d_{q,2q}}{(2q)!\times 2}\sum_{j=0}^{2q}(-1)^j{{2q}\choose{j}}[(q-j-1/2)^{2i}+(q-j-3/2)^{2i}].
\end{equation*}
Finally,  the identities (\ref{b6} ) and (\ref{b7}) follow by setting $c_{2i}=d_{i,2i}/((2i)!\times 2)$ and $c_{2i+1}=d_{i,2i+1}/(2i+1)!$, for $i=1,2,...,p$. 
\end{proof}

\begin{remark}
\label{rmk:2}
The approximations (\ref{b6}) and (\ref{b7}) are,  from the coefficients $c_i$ computed in Table \ref{tab:1}, equivalent to the central-difference approximation of the first derivative and the centered Bessel's formulae (see \cite[p.142 \& p.183]{hildebrand1974introduction} or \cite{chung2010computational,dahlquist2008numerical}).
\end{remark}

 \begin{remark}
\label{rmk:3}Formula (\ref{b6}) gives the finite difference approximations in  \cite{khan2000new}, writing
\begin{equation}
\label{b6b} 
u'(t_{n})=
\frac{u(t_{n+1/2})-u(t_{n-1/2})}{k}-\sum_{i=1}^pc_{2i+1}k^{2i}D(D_+D_-)^iu(t_{n}) +O(k^{2p+2}),
\end{equation}where
$$\sum_{i=1}^pc_{2i+1}k^{2i}D(D_+D_-)^iu(t_{n})=k^{-1}\sum_{i=1}^p\left[ c_{2i+1}\sum_{j=0}^{2i+1}(-1)^j{{2i+1}\choose{j}}u(t_{n+i-j+1/2})\right] .$$

\noindent
-~~ For $p=1$ we have
$$
\begin{aligned}
u'(t_{n})&=
\frac{u(t_{n+1/2})-u(t_{n-1/2})}{k}-\frac{1}{24} k^{2}D(D_+D_-)u(t_{n}) +O(k^{4})\\&=\frac{u(t_{n+1/2})-u(t_{n-1/2})}{k}-\frac{u(t_{n+3/2})-3u(t_{n+1/2})+3u(t_{n-1/2})-u(t_{n-3/2})}{24k}\\&\quad\quad+O(k^4).
\end{aligned}
$$

\noindent
-~~ For $p=2$ we have
$$
\begin{aligned}
u'(t_{n})&=
\frac{u(t_{n+1/2})-u(t_{n-1/2})}{k}-\frac{1}{24} k^{2}D(D_+D_-)u(t_{n})+\frac{18}{2^5 5!} k^{4}D(D_+D_-)^2u(t_{n})\\& +O(k^{6}),
\end{aligned}
$$and then
$$
\begin{aligned}
u'(t_{n})=
\frac{u(t_{n+1/2})-u(t_{n-1/2})}{k}+\frac{1}{1920k} \begin{bmatrix}
9& -125 & 330& -330 &125 & -9
\end{bmatrix}U^T_{n,5} +O(k^{6}),
\end{aligned}
$$
where $U^T_{n,5}$ is the transpose of the vector
$$U_{n,5}=\begin{bmatrix}
u(t_{n+5/2})~~&u(t_{n+3/2})~~&u(t_{n+1/2})~~&u(t_{n-1/2})~~&u(t_{n-3/2})~~&u(t_{n-5/2})
\end{bmatrix}.$$
\end{remark}

The following theorem gives a new form of centered finite difference formulae which is useful for efficient starting procedures of high order time-stepping schemes via deferred correction strategy \cite{koyaguerebo2019arbitrary,koyaguerebo2020unconditionally}.

\begin{theorem}[Interior centered approximations]
\label{theo:7}
Let $u\in C^{2p+3}\left( [a,b],X\right)$, where $p$ is a positive integer and $[a,b]$, $a<b$, is a real interval. Given a uniform  partition $a=\tau_0<\tau_1<...<\tau_{2p+1}=b$ of $[a,b]$, that is $\tau_n=a+nk$ with $k=(b-a)/(2p+1)$, and $\tau_{p+1/2}=(a+b)/2$, there exist reals $c_2^p,c_3^p,\cdots,c_{2p+1}^p$ such that 
\begin{equation}
\label{01} 
u'(\tau_{p+1/2})=
\frac{u(b)-u(a)}{b-a}-\frac{1}{b-a}\sum_{i=1}^pc^p_{2i+1}k^{2i+1}D(D_+D_-)^iu(\tau_{p+1/2}) +O(k^{2p+2}).
\end{equation}
and
\begin{equation}
\label{02} 
u(\tau_{p+1/2}) =\frac{u(b)+u(a)}{2}-\sum_{i=1}^pc^p_{2i}k^{2i}(D_+D_-)^iE u(\tau_{p+1/2})+O(k^{2p+2}),
\end{equation} 
Table \ref{tab:3} gives the coefficients ${c}^p_i$ for $p=1,2,3,4$.

\begin{table}[!ht] 
\caption{Coefficients of the approximations (\ref{01})-(\ref{02}) for $p=1,2,3,4$}
\label{tab:3}
 \centering      
\begin{tabular}{lllllllllll}  
\hline\noalign{\smallskip} 
   $p$& ${c}^p_2$ &~${c}^p_3$ &~~~ ${c}^p_4$ &~~~~${c}^p_5$& ~${c}^p_6$&~~${c}^p_7$&~~~~${c}^p_8$&~~~~~${c}^p_9$&~~~\\[1.5ex]
  1&  ~$\frac{9}{8}$        &~$\frac{9}{8}$       \\[1ex]
  2&  $\frac{25}{8}$        &$\frac{125}{24}$         &$\frac{125}{128}$      &$\frac{125}{128}$     \\[1ex]
  3&  $\frac{49}{8}$        &$\frac{343}{24}$         &$\frac{637}{128}$      &$\frac{13377}{1920}$     &$\frac{1029}{1024}$      &$\frac{1029}{1024}$      \\[1ex]
  4&  $\frac{81}{8}$        &$\frac{243}{8}$         &$\frac{1917}{128}$      &$\frac{17253}{640}$     &$\frac{7173}{1024}$      &$\frac{64557}{7168}$      &$\frac{32733}{32768}$   &$\frac{32733}{32768}$  \\[1ex]
\noalign{\smallskip}\hline
\end{tabular} 
\end{table}  
\end{theorem}

\begin{proof}By Taylor expansion we have
$$u(b)=u(a)-(b-a)u'(\tau_{p+1/2})+\sum_{i=1}^p \frac{(b-a)^{2i+1}}{2^{2i}(2i+1)!}u^{(2i+1)}(\tau_{p+1/2}), +O((b-a)^{2p+3}),$$
and
$$u(b)=-u(a)+2u(\tau_{p+1/2})+\sum_{i=1}^p \frac{(b-a)^{2i}}{2^{2i-1}(2i)!}u^{(2i)}(\tau_{p+1/2}) +O((b-a)^{2p+2}).$$
Substituting $b-a$ by $(2p+1)k$ in the summations, we deduce that
\begin{equation*}
u(b)=u(a)+(b-a)u'(\tau_{p+1/2})+\sum_{i=1}^p\frac{d^p_{1,2i+1}}{(2i+1)!}k ^{2i+1} u^{(2i+1)}(\tau_{p+1/2}) +O(k^{2p+3}),
\end{equation*}
and  
\begin{equation*}
u(b)=-u(a)+2u(\tau_{p+1/2})+\sum_{i=1}^p\frac{d^p_{1,2i}}{(2i)!}k ^{2i} u^{(2i)}(t_{p+1/2}) +O(k^{2p+2}),
\end{equation*}
where
$$d^p_{1,i}=2^{1-i}(2p+1)^i, ~\mbox{ for } i=1,\cdots,2p+1.$$ Proceeding exactly as in Theorem \ref{theo:1}, we obtain the real $d^p_{q,i}$ such that, for $q=1,...,p-1$ and $i=q+1,q+2,...,p$, we have 
\begin{equation*}
d^p_{q+1,2i+1}=d^p_{q,2i+1}-\frac{d^p_{q,2q+1}}{(2q+1)!}\sum_{j=0}^{2q+1}(-1)^j\binom{2q+1}{j}(q-j-1/2)^{2i+1},
\end{equation*}and
\begin{equation*}
d^p_{q+1,2i}=d^p_{q,2i}-\frac{d^p_{q,2q}}{(2q)!\times 2}\sum_{j=0}^{2q}(-1)^j\binom{2q}{j}\left[ (q-j-1/2)^{2i}+(q-j+1/2)^{2i}\right] .
\end{equation*}
Finally, $c^p_{2i}=d^p_{i,2i}/((2i)!\times 2)$ and $c^p_{2i+1}=d^p_{i,2i+1}/(2i+1)!$, for $i=1,2,...,p$. 
\end{proof}

The following finite difference formulae are useful for the construction of new time-stepping methods by applying the deferred correction method to backward or forward schemes.

\begin{theorem}(Forward-centered and backward-centered approximations)
\label{theo:ac1}
 There exists a sequence $\displaystyle \left\lbrace a_i \right\rbrace_{i\geq 2}$ and $\displaystyle \left\lbrace b_i \right\rbrace_{i\geq 2}$ of real numbers such that, for any function $u\in C^{p+1}\left( [0,T],X\right)$ and a partition $0=t_0<t_1<...<t_N=T$, $t_n=nk$, of $[0,T]$, we have
 \begin{equation}
\label{ab24}
u'(t_n)=\frac{u(t_{n+1})-u(t_n)}{k} -\sum_{i=2}^{p}a_{i}k^{i-1}D_-^{\tau (i)}(D_+D_-)^{\mu (i)}u(t_n)+O(k^{p}),
\end{equation}and
 \begin{equation}
\label{ab25}
u'(t_{n+1})=\frac{u(t_{n+1})-u(t_n)}{k} +\sum_{i=2}^{p}b_{i}k^{i-1}D_-^{\tau (i)}(D_+D_-)^{\mu (i)}u(t_{n+1})+O(k^{p}),
\end{equation}
for $\mu (p)+\tau(p)\leq n\leq N-\mu (p)$, where $\mu(i)$ and $\tau (i)$ are, respectively, the quotient and the remainder of the Euclidean division of the integer $i$ by 2, that is $i=2\mu (i)+\tau(i)$, $\tau(i) =0 \mbox{ or } 1$. The errors constants for the finite differences approximations (\ref{ab24})-(\ref{ab25}) are $a_{p+1}$ and $b_{p+1}$, respectively, and we have the relation  $a_2=b_2$, and $a_i=-b_i$, for $i=3,4,\cdots$.

Table \ref{tab:table1} gives the coefficients $a_i$, for $i=2,3,\cdots,11$. 
\begin{table}[h!] 
\caption{Table of coefficients, for differed correction backward Euler method.}
\label{tab:table1}
\begin{tabular}{llllllllllll}  
\hline\noalign{\smallskip} 
      & $a_2$ & $a_3$ & $a_4$ & \quad $a_5$ & \quad $a_6$ &  $a_7$ & \; $a_8$& \quad $a_9$& \quad $a_{10}$& \; $a_{11}$\\[2ex]
    &  $\frac{1}{2}$ & $\frac{1}{3!}$ &  $\frac{2}{4!}$ &  $-\frac{4}{5!}$  &  $-\frac{12}{6!}$ &$\frac{36}{7!}$ & $\frac{144}{8!}$ & $-\frac{576}{9!}$&$-\frac{2880}{10!}$&$\frac{14400}{11!}$\\[1ex]
\noalign{\smallskip}\hline
\end{tabular} 
\end{table}
\end{theorem}

\begin{proof} 
Taylor expansion of the function $u$ at order $p$ around $t=t_n$ gives
\begin{equation}
\label{ab26}u(t_{n+1})=u(t_n)+A_{1,1}ku'(t_n)+\sum_{i=2}^{p}A_{1,i}\frac{k^i}{i!}u^{(i)}(t_n)+O(k^{p+1}),
\end{equation}
where $A_{1,i}=1$, for $i=1, 2, 3,..., p$. Suppose that
\begin{equation}
\label{cyr26b}
\begin{aligned}
u&(t_{n+1})=u(t_n)+A_{1,1}ku'(t_n)+A_{2,2}k^2D_+D_-u(t_n)+A_{3,3}k^3D_-(D_+D_-)u(t_n)+...\\&+A_{q-1,q-1}k^{q-1}D_-^{\tau(q-1)}(D_+D_-)^{\mu(q-1)}u(t_n)+\sum_{i=q}^{p}A_{q-1,i}k^iu^{(i)}(t_n)+O(k^{p+1}),
\end{aligned}
\end{equation}
for an arbitrary integer $q\geq 2$, where (\ref{ab26}) is the formula for $q=2$. From (\ref{a25})-(\ref{a26}) and (\ref{a13}) we have
$$u^{(q)}(t_n)=D_-^{\tau(q)}(D_+D_-)^{\mu(q)}u(t_n)-\sum_{i=q+1}^\infty \frac{k^{i-q}}{i!} u^{(i)}(t_n)\sum_{j=0}^q(-1)^j{{q}\choose{j}}\left(\mu(q)-j \right)^{i},$$and it follows that
$$
\begin{aligned}
&\sum_{i=q}^{p}A_{q-1,i}k^iu^{(i)}(t_n)=A_{q-1,q}\frac{k^q}{q!}u^{(q)}(t_n)+\sum_{i=q+1}^{p}A_{q-1,i}\frac{k^i}{i!} u^{(i)}(t_n)\\&
=A_{q-1,q}\frac{k^q}{q!}D_-^{\tau(q)}(D_+D_-)^{\mu(q)}u(t_n)\\&+\sum_{i=q+1}^{p}\left(  A_{q-1,i}-\frac{A_{q-1,q}}{q!}\sum_{j=0}^q(-1)^j{{q}\choose{j}}\left(\mu(q)-j \right)^{i}\right) \frac{k^i}{i!} u^{(i)}(t_n)+O(k^{p+1}).
\end{aligned}
$$
Substituting the last identity in (\ref{cyr26b}), we deduce that 
\begin{equation*}
\begin{aligned}
u&(t_{n+1})=u(t_n)+ku'(t_n)+A_{2,2}k^2D_+D_-u(t_n)+A_{3,3}k^3D_-(D_+D_-)u(t_n)+...\\&+A_{q,q}k^{q}D_-^{\tau(q)}(D_+D_-)^{\mu(q)}u(t_n)+\sum_{i=q+1}^{p}A_{q,i}k^iu^{(i)}(t_n)+O(k^{p+1}),
\end{aligned}
\end{equation*}
where, for $q=2,3,\cdots, p$ we have
$$A_{q,q}=A_{q-1,q}$$
and
$$A_{q,i}=A_{q-1,i}-\displaystyle\frac{A_{q,q}}{q!}\sum_{j=0}^{q}(-1)^j\binom{q}{j}\left( \mu(q)-j\right) ^i,\mbox{for } i=q+1,q+2,...,p.$$ 
We can then deduce by induction on $q$ that formula (\ref{ab24}) holds with $a_i=A_{i,i}$, for $i=2,...,p$. The sequence $\displaystyle \left\lbrace b_i \right\rbrace_{i\geq 2}$ can be  obtained similarly.

\end{proof}

\begin{remark}
\label{rmk4}
The standard forward formula writes
 \begin{equation}
\label{ab27}
u'(t_n)=\frac{u(t_{n+1})-u(t_n)}{k} -\sum_{i=2}^{p}\frac{(-1)^i}{i}k^{i-1} D_+^iu(t_n)+O(k^{p}).
\end{equation}It can be obtained by substituting successively the derivative $u^{(2)}(t_n)$,  $u^{(3)}(t_n)$, ..., in (\ref{ab26}) by the  expansion (\ref{ab28}), and the standard backward formula writes
 \begin{equation}
\label{ab27b}
u'(t_{n+1})=\frac{u(t_{n+1})-u(t_n)}{k} +\sum_{i=2}^{p}\frac{1}{i}k^{i-1} D_-^iu(t_{n+1})+O(k^{p}),
\end{equation} and can be obtained from (\ref{ab29}). 
The errors constants in the new forward-centered and backward-centered formulae are smaller than for the standard forward and backward formulae (\ref{ab27}) and (\ref{ab27b}), respectively. For example, the error constant for an approximation of order 10 for $u'(t_n)$ by the formulae (\ref{ab27})-(\ref{ab27b}) is $1/11$ while the corresponding error constant for (\ref{ab24})-(\ref{ab25}) is $14400/11!$.
\end{remark}

More generally, we have the following result:
\begin{theorem}[General finite difference formulae]
\label{theo:a11} For an analytic function $u:\mathbb{R} \longrightarrow X$, given an integer $m$ and a real $k>0$, we can write, for any integer $p\geq m$ and a real $t$,
\begin{equation}
\label{ab30}
u^{(m)}(t)=k^{-m}\sum_{i=m}^p\sum_{\vert \alpha^i\vert=i}C_{\alpha^i}(k_i)^iD^{\alpha^i}u(t)+O(k^{p+1-m}),
\end{equation}
where $C_{\alpha^i}$ are constants,  $k_m=k$, $k_i=\varepsilon_i k$ (for $i\geq m+1$, where $\varepsilon_i>0$ is arbitrarily chosen), and each finite difference operator $D^{\alpha^i}$ is related to $k_i$ in the sense that
\begin{equation}
\label{ab30b}(k_i)^iD^{\alpha^i}u(t)=\sum_{j=0}^i(-1)^j{{i}\choose{j}}u\left( t+(\alpha^i_1-j)k_i\right), \mbox{ for } |\alpha^i|=i.
\end{equation}
\end{theorem}
\begin{proof}For a double index $\alpha^i=(\alpha^i_1,\alpha^i_2)$ such that $|\alpha^i|=i$ and a spacing  $k_i>0$, since $D^{\alpha^i}$ is related to $k_i>0$, we deduce from (\ref{ab30b}) and Theorem \ref{theo:a9} that 
\begin{equation}
\label{ab31}u^{(i)}(t)=D^{\alpha^i}u(t)-\sum_{l=i+1}^{\infty}\frac{(k_i)^{l-i}}{l!}u^{(l)}(t)\sum_{j=0}^i(-1)^j{{i}\choose{j}}\left(\alpha^i_1-j  \right)^l. 
\end{equation}
Therefore, we can choose one double index $\alpha^m$ such that $|\alpha^m|=m$ and deduce that
$$k^m u^{(m)}(t)=k^mD^{\alpha^m}u(t)-\sum_{l=m+1}^{\infty}\frac{k^{l}}{l!}u^{(l)}(t)\sum_{j=0}^m(-1)^j{{m}\choose{j}}\left(\alpha^m_1-j  \right)^l.$$
This identity can be written
\begin{equation}
\label{ab32}k^m u^{(m)}(t)=k^mD^{\alpha^m}u(t)+\sum_{l=m+1}^{\infty} C_{m+1,l}\frac{(k_{m+1})^{l}}{l!}u^{(l)}(t), 
\end{equation}
where $k_{m+1}=\varepsilon_{m+1}k$, for a real $\varepsilon_{m+1}>0$ arbitrarily chosen, and
$$C_{m+1,l}=-(\varepsilon_{m+1})^{-l}\sum_{j=0}^m(-1)^j{{m}\choose{j}}\left(\alpha^m_1-j  \right)^l, \mbox{ for } l\geq m+1.$$
Next, we choose one double index $\alpha^{m+1}$ such that $|\alpha^{m+1}|=m+1$ and substitute the identity (\ref{ab31}) for $i=m+1$ into (\ref{ab32}) to obtain
\begin{equation}
\label{ab33}
\begin{aligned}
k^m u^{(m)}(t)=k^mD^{\alpha^m}u(t)&+C_{m+1,m+1}(k_{m+1})^{m+1}D^{\alpha^{m+1}}u(t)\\&+\sum_{l=m+2}^{\infty} C_{m+2,l}\frac{(k_{m+2})^{l}}{l!}u^{(l)}(t),
\end{aligned} 
\end{equation}
where $k_{m+2}=\varepsilon_{m+2}k_{m+1}$, for a real $\varepsilon_{m+2}>0$ arbitrarily chosen, and, for $l\geq m+2$,
$$C_{m+2,l}=(\varepsilon_{m+2})^{-l}\left(C_{m+1,l}-\frac{C_{m+1,m+1}}{(m+1)!} \sum_{j=0}^{m+1}(-1)^j{{m+1}\choose{j}}\left(\alpha^{m+1}_1-j  \right)^l\right). $$
This procedure is repeated until obtaining the expected order of accuracy.
\end{proof}

\begin{remark} As a simple application of Theorem \ref{theo:a11}, the standard central difference for the second derivative (see, e.g., \cite[Formulae (3.3.10)-(3.3.11)]{chung2010computational}) can be obtained as follows:
We choose $m=1$ in formula (\ref{a26}) and obtain
\begin{equation}
\label{ab34}
k^2u"(t_n)=k^2(D_+D_-)u(t_n)-2\sum_{i=2}^{\infty}\frac{k^{2i}}{(2i)!}u^{(2i)}(t_{n}),
\end{equation}
which is the second order approximation of $u"(t_n)$ with error constant $K_2=-1/12$. The same formula for $m=2$ gives
\begin{equation*}
k^4u^{(4)}(t_n)=k^4(D_+D_-)^2u(t_n)-\sum_{i=3}^{\infty} \frac{k^{2i}}{(2i)!}u^{(2i)}(t_{n})\sum_{j=0}^{4}(-1)^j\binom{4}{j}(2-j)^{2i}.
\end{equation*}
Substituting the last identity in (\ref{ab34}), we deduce that 
\begin{equation*}
\begin{aligned}
k^2u"(t_n)=&k^2(D_+D_-)u(t_n)-\frac{2k^4}{4!}(D_+D_-)^2u(t_n)\\&+\sum_{i=3}^{\infty}\left(-2+\frac{2}{4!}\sum_{j=0}^{4}(-1)^j\binom{4}{j}(2-j)^{2i} \right) \frac{k^{2i}}{(2i)!}u^{(2i)}(t_{n}).
\end{aligned}
\end{equation*}
The last formula gives the approximation of order 4 for $u"(t_n)$ with error constant 
$$K_4=\left(-2+\frac{2}{4!}\sum_{j=0}^{4}(-1)^j\binom{4}{j}(2-j)^{6} \right) \frac{1}{6!}=\frac{1}{90}.$$
The arbitrary high order central difference can be obtained by continuing the procedure.
\end{remark}

\section{Numerical test}
This section deals with a comparison between the standard finite difference formulae and the new formulae  obtained in Theorem \ref{theo:7} and \ref{theo:ac1}. The comparisons address the numerical differentiation of the functions  $u(x)=\sin(100\pi x)$ and $u(x)=\sin(1000\pi x)$ which are taken from the list of tests functions in \cite{khan2000new}. For the classical finite difference formulae we just select the backward formulae of order 6 and 10, denoted $B6$ and $B10$, respectively. For the new finite difference formulae we choose the backward-centered formulae of order 6 and 10, denoted $BC6$ and $BC10$, respectively, and the interior-centered formulae of order 6 and 10, denoted $IC6$ and $IC10$, respectively. We drop the standard forward finite difference formula since it reaches the same accuracy as the backward formula (for a same order of approximation). The standard centered finite difference formula has the accuracy of the interior-centered formula so that we choose to not show it. Finally, the forward-centered formula reaches the same accuracy as the backward-centered formula.

Figure \ref{fig:2} shows that each of the finite difference formulae choosen gives a good approximate derivative of the functions considered. The accuracy of the approximations are related to both the order of accuracy of the corresponding formula and its error constant. Moreover, the new formulae are less prone to floating point error when the approximation reaches machine accuracy.

\begin{figure}[!ht]
\centering
{
    \includegraphics[width=.49\textwidth]{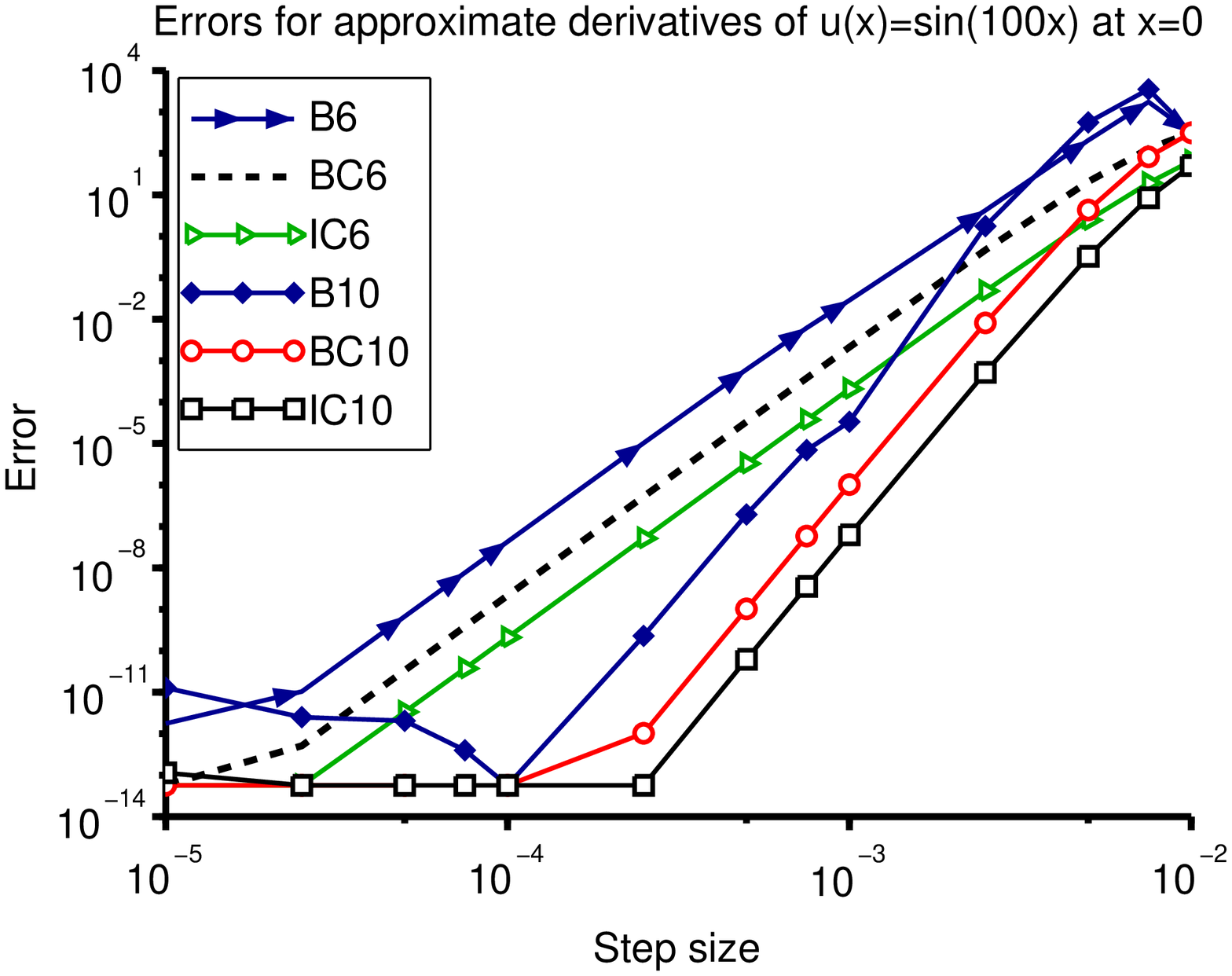}
  \includegraphics[width=.49\textwidth]{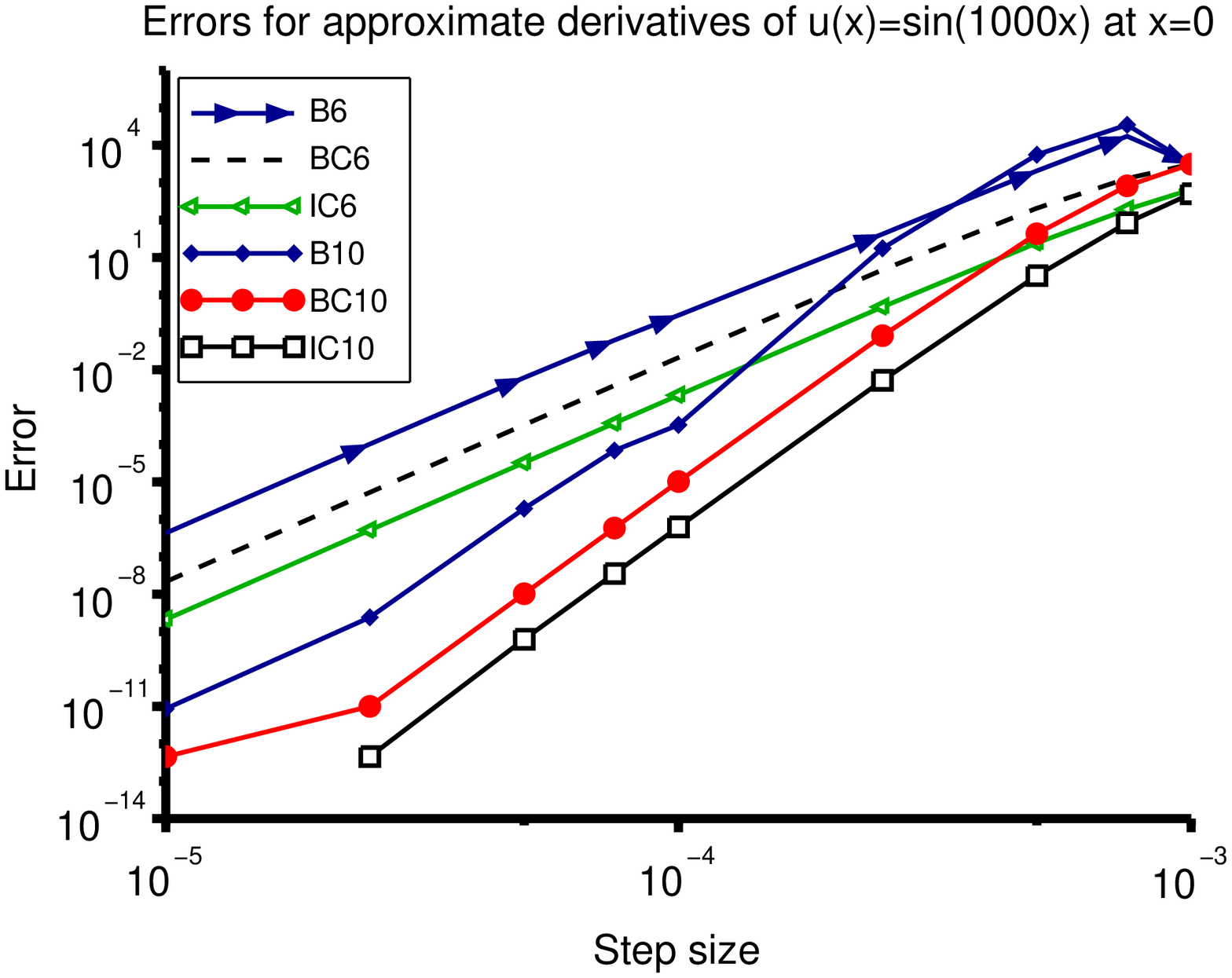}
   \label{fig:subfig2}
}
\caption[]{Graphs of absolute error for the numerical derivative of $u(x)=\sin(100\pi x)$ (left) and $u(x)=\sin(1000\pi x)$ (right) at $x=0$ with $B6$, $B10$, $BC6$, $BC10$, $IC6$ and $IC10$.}
\label{fig:2}
\end{figure}


%
%

\bibliographystyle{elsarticle-num}       
\bibliography{references}   


\end{document}